\documentclass{amsart}
\usepackage{mathptmx, amssymb, enumerate, colonequals, bm}
\usepackage{hyperref}

\newtheorem{theorem}{Theorem}[section]
\newtheorem{lemma}[theorem]{Lemma}
\newtheorem{proposition}[theorem]{Proposition}

\theoremstyle{definition}
\newtheorem{definition}[theorem]{Definition}
\newtheorem{remark}[theorem]{Remark}
\newtheorem{question}[theorem]{Question}
\newtheorem*{notation}{Notation}
\numberwithin{equation}{theorem}

\def\ge{\geqslant}
\def\le{\leqslant}
\def\la{\langle}
\def\ra{\rangle}

\def\phi{\varphi}

\def\tilde{\widetilde}
\def\to{\longrightarrow}

\def\fpt{\operatorname{fpt}}
\def\height{\operatorname{height}\,}
\def\inlex{\operatorname{in_{lex}}}
\def\ord{\operatorname{ord}\,}
\def\Ass{\operatorname{Ass}}
\def\Cl{\operatorname{Cl}}

\def\fraka{\mathfrak{a}}

\def\frakm{\mathfrak{m}}

\def\frakp{\mathfrak{p}}
\def\frakq{\mathfrak{q}}

\def\bsx{{\boldsymbol{x}}}

\def\FF{\mathbb{F}}
\def\KK{\mathbb{K}}
\def\NN{\mathbb{N}}
\def\PP{\mathbb{P}}
\def\QQ{\mathbb{Q}}
\def\ZZ{\mathbb{Z}}

\begin{document}
\title{Hankel determinantal rings have rational singularities}

\author{Aldo Conca}
\address{Dipartimento di Matematica, Universit\`a di Genova, Via Dodecaneso 35, I-16146 Genova, Italy}
\email{conca@dima.unige.it}

\author{Maral Mostafazadehfard}
\address{IMPA, Estrada Dona Castorina 110, 22460-320 Rio de Janeiro, Brazil}
\email{maralmostafazadehfard@gmail.com}

\author{Anurag K. Singh}
\address{Department of Mathematics, University of Utah, 155 South 1400 East, Salt Lake City, UT~84112, USA}
\email{singh@math.utah.edu}

\author{Matteo Varbaro}
\address{Dipartimento di Matematica, Universit\`a di Genova, Via Dodecaneso 35, I-16146 Genova, Italy}
\email{varbaro@dima.unige.it}

\thanks{A.C.~and M.V.~were supported by GNSAGA-INdAM; M.M.~was supported by post doctoral fellowships CNPq-PDE 204593/2014-0 and CAPES-PNPD; A.K.S.~was supported by NSF grant DMS~1500613, and his research stay at Universit\`a di Genova was partially supported by the Simons Foundation and by the Mathematisches Forschungsinstitut Oberwolfach; M.V.~was also supported by the \lq\lq program for abroad research activity of the University of Genoa\rq\rq, Pr. N, 100021-2016-MV-ALTROUNIGE$\_$001.\\
\phantom{..}\quad We are grateful to the referee for a careful reading of the manuscript, and for helpful suggestions.}

\begin{abstract}
Hankel determinantal rings, i.e., determinantal rings defined by minors of Hankel matrices of indeterminates, arise as homogeneous coordinate rings of higher order secant varieties of rational normal curves; they may also be viewed as linear specializations of generic determinantal rings. We prove that, over fields of characteristic zero, Hankel determinantal rings have rational singularities; in the case of positive prime characteristic, we prove that they are $F$-pure. Independent of the characteristic, we give a complete description of the divisor class groups of these rings, and show that each divisor class group element is the class of a maximal Cohen-Macaulay module.
\end{abstract}
\maketitle

%%%%%%%%%%%%%%%%%%%%%%%%%%%%%%%%%%%%%%%%%%%%%%%%%%%%%%%%%%%%%%%%%%%%%%%%%%%%%%%%%%
\section*{Introduction}
%%%%%%%%%%%%%%%%%%%%%%%%%%%%%%%%%%%%%%%%%%%%%%%%%%%%%%%%%%%%%%%%%%%%%%%%%%%%%%%%%%

Throughout this paper, by a \emph{Hankel matrix}, we mean a matrix of the form
\[
H\colonequals\begin{pmatrix}
x_1 & x_2 & x_3 & \cdots & x_s\\
x_2 & x_3 & \cdots & \cdots & x_{s+1}\\
x_3 & \cdots & \cdots & \cdots & x_{s+2}\\
\vdots & \vdots & \vdots & \vdots & \vdots\\
x_r & \cdots & \cdots & \cdots & x_{s+r-1}
\end{pmatrix},
\]
where $x_1,\dots,x_{s+r-1}$ are indeterminates over a field $\FF$. By a \emph{Hankel determinantal ring} we mean a ring of the form
\[
\FF[x_1,\dots,x_{s+r-1}]/I_t(H),
\]
where~$1\le t\le\min\{r,s\}$, and $I_t(H)$ is the ideal generated by the size $t$ minors of~$H$. These rings arise as homogeneous coordinate rings of higher order secant varieties of rational normal curves, see for example Room's 1938 study \cite[Chapter~11.7]{Room}.

We prove that Hankel determinantal rings over fields of characteristic zero have rational singularities, Theorem~\ref{theorem:ratsing}. In particular, higher order secant varieties of rational normal curves have rational singularities. Theorem~\ref{theorem:ratsing} may be compared with corresponding statements for generic determinantal rings, and those defined by minors of symmetric matrices of indeterminates or by pfaffians of skew-symmetric matrices of indeterminates: in characteristic zero, these are all invariant rings of linearly reductive classical groups acting on polynomial rings, and hence are pure subrings of polynomial rings. By Boutot's theorem~\cite{Boutot}, it then follows that they have rational singularities. We do not know if Hankel determinantal rings, in general, arise as invariant rings for group actions on polynomial rings, or if they are pure subrings of polynomial rings. However, for $t\ge 3$, we show that they are not pure subrings of the polynomial rings in which they are naturally embedded, see Proposition~\ref{proposition:not:pure:subrings}. Our proof of rational singularities is via reduction modulo~$p$ methods, using Smith's theorem \cite{Smith:ratsing} that rings of $F$-rational type have rational singularities. 

We compute the divisor class groups of Hankel determinantal rings: the group is finite cyclic, in particular, the rings are $\QQ$-Gorenstein, and we show that each divisor class group element corresponds to a rank one maximal Cohen-Macaulay module, see Theorem~\ref{theorem:class:group}.

We also prove that Hankel determinantal rings over fields of positive characteristic are~$F$-pure, Theorem~\ref{theorem:fpure}. Finally, for $R$ a Hankel determinantal ring with homogeneous maximal ideal $\frakm_R$, we compute the $F$-pure threshold of~$\frakm_R$ in $R$, and of its defining ideal~$I_t(H)$ in the ambient polynomial ring, see Theorems~\ref{theorem:fpt1} and~\ref{theorem:fpt2}. 

%%%%%%%%%%%%%%%%%%%%%%%%%%%%%%%%%%%%%%%%%%%%%%%%%%%%%%%%%%%%%%%%%%%%%%%%%%%%%%%%%%
\section{Generalities}
%%%%%%%%%%%%%%%%%%%%%%%%%%%%%%%%%%%%%%%%%%%%%%%%%%%%%%%%%%%%%%%%%%%%%%%%%%%%%%%%%%

By a result of Gruson and Peskine,~\cite[Lemme~2.3]{Gruson-Peskine}, every Hankel determinantal ring is isomorphic to one where the defining ideal $I_t(H)$ is generated by the maximal sized minors of a Hankel matrix; alternatively see \cite[Proposition~7]{JWatanabe} or \cite[Corollary~2.2(b)]{Conca}. In view of this, we will henceforth work with Hankel determinantal rings of the form
\[
R\colonequals\FF[x_1,\dots,x_{n+t-1}]/I_t(H),
\]
where $H$ is a $t\times n$ Hankel matrix and $t\le n$; except where stated otherwise, $H$ will denote such a~matrix.

Consider the \emph{generic} determinantal ring
\[
B\colonequals\FF[Y]/I_t(Y),
\]
where $Y$ is a~$t\times n$ matrix of indeterminates, and $I_t(Y)$ the ideal generated by its size $t$ minors. The $(t-1)(n-1)$ elements $Y_{i,j+1}-Y_{i+1,j}$ are readily seen to be part of a system of parameters for $B$, and specializing these to $0$ gives a ring isomorphic to~$R$. Since~$B$ is Cohen-Macaulay by \cite{Eagon-Northcott, Hochster-Eagon}, so is the ring $R$, and the Eagon-Northcott complex provides a minimal free resolution of $R$. It follows as well that
\[
\dim R\ =\ 2t-2,
\]
and hence that
\[
\height I_t(H)\ =\ n-t+1.
\]
The elements $x_1,\dots,x_{t-1},x_{n+1},\dots,x_{n+t-1}$ are a homogeneous system of parameters for~$R$, and the socle modulo this system of parameters is spanned by the degree $t-1$ monomials in~$x_t,\dots,x_n$. In particular, the ring $R$ has $a$-invariant
\[
a(R)\ =\ 1-t.
\]
The multiplicity of the ring $R$ is
\[
e(R)\ =\ \binom{n}{t-1},
\]
as may be seen directly from the above discussion, or obtained using the multiplicity of generic determinantal rings, e.g., \cite[Proposition~2.15]{Bruns-Vetter}.

The ring $R$ is a normal domain, see for example,~\cite[Proposition~8]{JWatanabe}; it is Gorenstein precisely when~$t=n$. The ideal $I_t(H)$ is a set-theoretic complete intersection by~\cite{Valla}. The singular locus of $R$ is defined by the image of $I_{t-1}(H)$, see \cite[Theorem~1.56]{Iarrobino-Kanev}. For secant varieties of smooth curves in general, we mention \cite{Vermeire} and the references therein. 

\begin{notation}
Given a matrix $X$, we use ${[a_1\ \dots\ a_r\mid b_1\ \dots\ b_r]}_X$ to denote the determinant of the submatrix of $X$ with rows $a_1,\dots,a_r$ and columns $b_1,\dots,b_r$. We omit the subscript whenever the matrix is clear from the context. 
\end{notation}

%%%%%%%%%%%%%%%%%%%%%%%%%%%%%%%%%%%%%%%%%%%%%%%%%%%%%%%%%%%%%%%%%%%%%%%%%%%%%%%%%%
\section{Rational singularities}
%%%%%%%%%%%%%%%%%%%%%%%%%%%%%%%%%%%%%%%%%%%%%%%%%%%%%%%%%%%%%%%%%%%%%%%%%%%%%%%%%%

In proving that Hankel determinantal rings of characteristic zero have rational singularities, we will use the following description: A $2\times n$ Hankel determinantal ring over a field~$\FF$ is readily seen to be isomorphic to the $n$-th Veronese subring of a polynomial ring~$\FF[u,v]$, where the Hankel matrix maps entrywise to
\[
\begin{pmatrix}
u^n & u^{n-1}v & u^{n-2}v^2 & \cdots & uv^{n-1}\\
u^{n-1}v & u^{n-2}v^2 & \cdots & \cdots & v^n
\end{pmatrix}.
\]
This is the homogeneous coordinate ring of the rational normal curve $C_n$ in $\PP^n$; as it is a Veronese subring, it is a pure subring of $\FF[u,v]$, independent of the characteristic of $\FF$. 

A $3\times n$ Hankel determinantal ring is the homogeneous coordinate ring of the secant variety of the rational normal curve $C_{n+1}$ in $\PP^{n+1}$; it is isomorphic to the subring of the polynomial ring $\FF[u_1,u_2,v_1,v_2]$, where the Hankel matrix maps entrywise to the matrix
\[
\begin{pmatrix}
u_1^{n+1}+u_2^{n+1} & u_1^nv_1^{\phantom.}+u_2^nv_2^{\phantom.} & u_1^{n-1}v_1^2+u_2^{n-1}v_2^2 & \cdots & u_1^2v_1^{n-1}+u_2^2v_2^{n-1}
\\[0.4em]
u_1^nv_1^{\phantom.}+u_2^nv_2^{\phantom.} & u_1^{n-1}v_1^2+u_2^{n-1}v_2^2 & \cdots & \cdots & u_1^{\phantom.}v_1^n+u_2^{\phantom.}v_2^n
\\[0.4em]
u_1^{n-1}v_1^2+u_2^{n-1}v_2^2 & \cdots & \cdots & \cdots & v_1^{n+1}+v_2^{n+1}
\end{pmatrix}.
\]

More generally, the Hankel determinantal ring $R\colonequals\FF[x_1,\dots,x_{n+t-1}]/I_t(H)$ is the homogeneous coordinate ring of the order $t-2$ secant variety of the rational normal curve~$C_{n+t-2}$ in $\PP^{n+t-2}$, see for example~\cite[Section~4]{Eisenbud}. Specifically, we claim that $R$ is isomorphic to the $\FF$-subalgebra of the polynomial ring
\[
S\colonequals\FF[u_1,\dots,u_{t-1},v_1,\dots,v_{t-1}]
\]
generated by the elements
\begin{equation}
\label{equation:secant}
h_i\ \colonequals\ u_1^{n+t-2-i}v_1^i\ +\ u_2^{n+t-2-i}v_2^i\ +\ \cdots\ +\ u_{t-1}^{n+t-2-i}v_{t-1}^i,
\quad\text{ for }\ 0\le i\le n+t-2.
\end{equation}

To see this, consider the $\FF$-algebra homomorphism $\phi\colon\FF[x_1,\dots,x_{n+t-1}]\to S$ defined by $\phi(x_i)=h_{i-1}$ for each $i$. Note that $\phi$ maps the Hankel matrix of indeterminates $H$ to a matrix $M$ that is Hankel in the elements~$h_i$. As $M$ is the sum of $t-1$ matrices of the form
\[
\begin{pmatrix}
u^{n+t-2} & u^{n+t-3}v & u^{n+t-4}v^2 & \cdots & u^{t-1}v^{n-1}\\
u^{n+t-3}v & u^{n+t-4}v^2 & \cdots & \cdots & u^{t-2}v^n\\
u^{n+t-4}v^2 & \cdots & \cdots & \cdots & u^{t-3}v^{n+1}\\
\vdots & \vdots & \vdots & \vdots & \vdots\\
u^{n-1}v^{t-1} & \cdots & \cdots & \cdots & v^{n-t+2}
\end{pmatrix},
\]
each having rank~$1$, it follows that the rank of $M$ is at most $t-1$, i.e., that $I_t(M)=0$. Hence~$\phi$ induces a homomorphism $\tilde{\phi}\colon R\to S$. Since $R$ is a domain of dimension $2t-2$, which is also the dimension of $S$, it follows that $\tilde{\phi}$ is injective.

\begin{theorem}
\label{theorem:ratsing}
Let $R=\FF[x_1,\dots,x_{n+t-1}]/I_t(H)$, where $\FF$ is a field, and $H$ is a $t\times n$ Hankel matrix. If $\FF$ has characteristic zero, then $R$ has rational singularities. If $\FF$ is a field of positive characteristic $p$, with $p\ge t$, then $R$ is $F$-rational.
\end{theorem}

It follows that if $\FF$ has characteristic $p\ge t$, then the ring $R$ has rational singularities in the sense of \cite[Definition~1.3]{Kovacs}; see \cite[Corollary~1.12]{Kovacs}.

\begin{proof}
It suffices to prove the positive characteristic assertion in the theorem: it then follows that $R$ is of $F$-rational type for $\FF$ of characteristic zero, and then by \cite[Theorem~4.3]{Smith:ratsing} that $R$ has rational singularities.

Let $\FF$ be a field of characteristic $p\ge t$, and assume $t\ge 2$. Using \cite[Theorem~4.7]{HH:JAG} and the preceding remark in that paper, it suffices to prove that the ideal generated by one choice of a homogeneous system of parameters for~$R$ is tightly closed. Set
\[
S\colonequals\FF[u_1,\dots,u_{t-1},v_1,\dots,v_{t-1}],
\]
i.e., $S$ is a polynomial ring in $2t-2$ indeterminates, and identify $R$ with the subring generated by the elements $h_0,\dots,h_{n+t-2}$ as in~\eqref{equation:secant}. The elements
\[
h_0,\dots,h_{t-2},h_n,\dots,h_{n+t-2}
\]
form a homogeneous system of parameters for~$R$. Let $\fraka$ be the ideal of $R$ generated by these elements; it suffices to show that~$\fraka$ is tightly closed. Note that $h_i$ belongs to the ideal~$(u_1^n,u_2^n,\dots,u_{t-1}^n)S$ for $0\le i\le t-2$, and to~$(v_1^n,v_2^n,\dots,v_{t-1}^n)S$ for $n\le i\le n+t-2$, so
\[
\fraka S\ \subseteq\ (u_1^n,u_2^n,\dots,u_{t-1}^n,v_1^n,v_2^n,\dots,v_{t-1}^n)S. 
\]

The socle of $R/\fraka$ is the vector space spanned by the images of the elements
\[
h_{i_1}h_{i_2}\cdots h_{i_{t-1}}
\quad\text{ where }\ t-1\le i_1\le i_2\le\cdots\le i_{t-1}\le n-1.
\]
Suppose that a linear combination of the above elements, say
\[
r\colonequals\sum\lambda_{i_1 i_2 \cdots i_{t-1}} h_{i_1}h_{i_2}\cdots h_{i_{t-1}}
\quad\text{ where }\ \lambda_{i_1 i_2 \cdots i_{t-1}}\in\FF,
\] 
belongs to $\fraka^*$, i.e., to the tight closure of $\fraka$ in $R$. Since $R\subset S$ is an inclusion of domains, it then follows from the definition of tight closure that $r\in (\fraka S)^*$. But $(\fraka S)^*=\fraka S$ since $S$ is regular, implying that $r\in\fraka S$, and hence that
\[
r\ \in\ (u_1^n,u_2^n,\dots,u_{t-1}^n,v_1^n,v_2^n,\dots,v_{t-1}^n)S. 
\]
We claim that this occurs only when each coefficient $\lambda_{i_1 i_2 \cdots i_{t-1}}$ equals $0$; it then follows that~$r=0$, i.e., that $\fraka$ is tightly closed, as desired.

We first illustrate the proof of the claim when $t=3$. In this case, the ring $R$ may be identified with the $\FF$-subalgebra of $S=\FF[u_1,u_2,v_1,v_2]$ generated by the elements
\[
h_i\ =\ u_1^{n+1-i}v_1^i + u_2^{n+1-i}v_2^i
\quad\text{ where }\ 0\le i\le n+1.
\]
Suppose
\[
r\ =\sum_{2\le i_1\le i_2\le n-1} \lambda_{i_1i_2} h_{i_1}h_{i_2}\ \in\ (u_1^n,u_2^n,v_1^n,v_2^n)S.
\]
Fix $k_1,k_2$ with $2\le k_1\le k_2\le n-1$, and consider the coefficient of $u_1^{n+1-k_1}v_1^{k_1} u_2^{n+1-k_2}v_2^{k_2}$ in the expression above, i.e., in
\[
\sum\lambda_{i_1i_2} h_{i_1}h_{i_2}\ =\ \sum\lambda_{i_1i_2} (u_1^{n+1-i_1}v_1^{i_1} + u_2^{n+1-i_1}v_2^{i_1}) (u_1^{n+1-i_2}v_1^{i_2} + u_2^{n+1-i_2}v_2^{i_2}).
\]
This coefficient is $\lambda_{k_1k_2}$ if $k_1<k_2$, and it equals $2\lambda_{k_1k_1}$ if $k_1=k_2$. Since the characteristic of~$\FF$ is~$p\ge 3$, and $r\in (u_1^n,u_2^n,v_1^n,v_2^n)S$, it follows that each coefficient must be $0$ as claimed.

We now turn to the general case: suppose
\[
r\ =\sum\lambda_{i_1 i_2 \cdots i_{t-1}} h_{i_1}h_{i_2}\cdots h_{i_{t-1}}\ \in\ (u_1^n,u_2^n,\dots,u_{t-1}^n,v_1^n,v_2^n,\dots,v_{t-1}^n)S,
\]
where the sum is over indices with $t-1\le i_1\le i_2\le\cdots\le i_{t-1}\le n-1$. Let $k_1,\dots,k_{t-1}$ be integers with
\[
t-1\le k_1\le k_2\le\cdots\le k_{t-1}\le n-1.
\]
The coefficient of 
\[
u_1^{n+t-2-k_1}v_1^{k_1} u_2^{n+t-2-k_2}v_2^{k_2} \cdots u_{t-1}^{n+t-2-k_{t-1}}v_{t-1}^{k_{t-1}}
\]
in $r$ is $c\lambda_{k_1 k_2 \cdots k_{t-1}}$ where $c$ is a product of positive integers, each less than $t$. Hence $c\neq0$ in~$\FF$, and so it follows that each coefficient is $0$.
\end{proof}

While the description in terms of higher secant varieties shows that every Hankel determinantal ring is a subring of a polynomial ring, it is not in general a pure subring of that polynomial ring, as we show next; recall that a ring homomorphism $R\to S$ is \emph{pure} if
\[
R\otimes_RM\to S\otimes_RM
\]
is injective for each $R$-module $M$.

\begin{proposition}
\label{proposition:not:pure:subrings}
Let $R$ be a $t\times n$ Hankel determinantal ring, regarded as the $\FF$-subalgebra of the polynomial ring $S=\FF[u_1,\dots,u_{t-1},v_1,\dots,v_{t-1}]$, generated by the elements $h_i$ as in~\eqref{equation:secant}. If $t\ge 3$, then $R$ is not a pure subring of~$S$.
\end{proposition}

\begin{proof}
Let $\frakm_R$ denote the homogeneous maximal ideal of $R$. The expansion of this ideal to $S$ is contained in the height $t$ ideal
\[
(u_1-v_1,\ \dots,\ u_{t-1}-v_{t-1},\ v_1^{n+t-2}+\cdots+v_{t-1}^{n+t-2})S.
\]
Since $\height\frakm_RS\le t<2t-2=\dim S$, the Hartshorne-Lichtenbaum Vanishing Theorem, for example \cite[Theorem~14.1]{24hours}, implies that
\[
H^{2t-2}_{\frakm_R}(S)\ =\ H^{2t-2}_{\frakm_RS}(S)\ =\ 0.
\]
If $R\to S$ is pure, the injectivity of $H^{2t-2}_{\frakm_R}(R)\to H^{2t-2}_{\frakm_R}(S)$ implies that $H^{2t-2}_{\frakm_R}(R)=0$, which is a contradiction since $\dim R=2t-2$.
\end{proof}

\begin{remark}
\label{remark:not:pure:subrings}
Being a pure subring of a polynomial ring is a stronger property than having rational singularities, or even having $F$-regular type; the hypersurface in~\cite[Theorem~5.1]{Singh-Swanson} has $F$-regular type, but is not a pure subring of a polynomial ring.
\end{remark}

\begin{question}
\label{question:pure:subrings} 
Is every Hankel determinantal ring a pure subring of a polynomial ring? 
\end{question}

%%%%%%%%%%%%%%%%%%%%%%%%%%%%%%%%%%%%%%%%%%%%%%%%%%%%%%%%%%%%%%%%%%%%%%%%%%%%%%%%%%
\section{The divisor class group}
%%%%%%%%%%%%%%%%%%%%%%%%%%%%%%%%%%%%%%%%%%%%%%%%%%%%%%%%%%%%%%%%%%%%%%%%%%%%%%%%%%

Consider the Hankel determinantal ring $R=\FF[x_1,\dots,x_{n+t-1}]/I_t(H)$, where $\FF$ is a field. To avoid some trivialities, we assume throughout this section that $n\ge t\ge 2$. Set $\frakp$ to be the ideal of $R$ generated by the maximal minors of the first $t-1$ rows of $H$, i.e.,
\begin{equation}
\label{equation:p}
\frakp\colonequals I_{t-1}\begin{pmatrix}
x_1 & x_2 & x_3 & \cdots & x_n\\
x_2 & x_3 & \cdots & \cdots & x_{n+1}\\
x_3 & \cdots & \cdots & \cdots & x_{n+2}\\
\vdots & \vdots & \vdots & \vdots & \vdots\\
x_{t-1} & \cdots & \cdots & \cdots & x_{n+t-2}
\end{pmatrix}.
\end{equation}
The ring $R/\frakp$ may be identified with the polynomial ring in the indeterminate $x_{n+t-1}$ over a size $(t-1)\times n$ Hankel determinantal ring; it follows that $R/\frakp$ is an integral domain of dimension $2t-3$, and hence that $\frakp$ is a prime ideal of height $1$.

For each integer $k$ with $1\le k\le n-t+2$, set $\frakp^{\la k\ra}$ to be the ideal of $R$ as below,
\[
\frakp^{\la k\ra}\colonequals I_{t-1}\begin{pmatrix}
x_1 & x_2 & x_3 & \cdots & x_{n-k+1}\\
x_2 & x_3 & \cdots & \cdots & x_{n-k+2}\\
x_3 & \cdots & \cdots & \cdots & x_{n-k+3}\\
\vdots & \vdots & \vdots & \vdots & \vdots\\
x_{t-1} & \cdots & \cdots & \cdots & x_{n-k+t-1}
\end{pmatrix}.
\]
Note that $\frakp^{\la 1\ra}=\frakp$, and that the ideal $\frakp^{\la n-t+2\ra}$ is principal. With this notation, we prove:

\begin{theorem}
\label{theorem:class:group}
Consider the Hankel determinantal ring $R\colonequals\FF[x_1,\dots,x_{n+t-1}]/I_t(H)$, for~$\FF$ a field, and $t\ge 2$. Then the divisor class group of $R$ is cyclic of order $n-t+2$, generated by the ideal $\frakp$ as in~\eqref{equation:p}. The symbolic powers of $\frakp$ are
\[
\frakp^{(k)}=\frakp^{\la k\ra}
\quad\text{ for }\ 1\le k\le n-t+2.
\]
Moreover, each of these is a maximal Cohen-Macaulay $R$-module.
\end{theorem}

We need a number of preliminary results.

\begin{lemma}
\label{lemma:identity}
Let $Y$ be an $m\times n$ matrix with entries in a commutative ring. Assume that $Y$ has rank less than $t$. 

\begin{enumerate}[\ \rm(1)]
\item For every choice of row and column indices, one has
\begin{multline*}
[a_1\ \dots\ a_{t-1} \mid b_1\ \dots\ b_{t-1}] \times [c_1\ \dots\ c_{t-1} \mid d_1 \ \dots\ d_{t-1}] \\
=\
[a_1\ \dots\ a_{t-1} \mid d_1\ \dots\ d_{t-1}] \times [c_1\ \dots\ c_{t-1} \mid b_1 \ \dots\ b_{t-1}].
\end{multline*}

\item Let $Y(a,b)$ denote the submatrix of $Y$ with row indices $\le a$, and column indices $\le b$. Then, for all $a<m$ and $b<n$, one has
\[
I_{t-1}(Y(a,b+1))\ I_{t-1}(Y(a+1,b))\ =\ I_{t-1}(Y(a,b))\ I_{t-1}(Y(a+1,b+1)).
\] 
\end{enumerate}
\end{lemma}

\begin{proof}
First note that (2) follows immediately from (1). To prove (1), we may assume right away that the underlying ring is $B\colonequals\ZZ[X]/I_t(X)$, with $X$ an $m\times n$ matrix of indeterminates, and that $Y$ is the image of $X$ in $B$. Since $B$ is a domain, it suffices to verify the displayed identity in the fraction field $\KK$ of $B$. Consider the linear map
\[
\phi\colon\KK^{m}\to\KK^{n}
\]
given by the image of $Y$. The map $\phi$ has rank less than $t$, so the exterior power
\[
\Lambda^{t-1}\phi\colon\Lambda^{t-1} \KK^{m}\to\Lambda^{t-1} \KK^{n}
\]
is a linear map of rank at most $1$. For rows $i_1,\dots,i_{t-1}$ and columns $j_1,\dots,j_{t-1}$ of $\phi$, the corresponding matrix entry of $\Lambda^{t-1} \phi$ is the determinant
\[
{[i_1\ \dots\ i_{t-1} \mid j_1\ \dots\ j_{t-1}]}_Y.
\]
The required identity is now immediate from the fact that the size $2$ minors of the matrix for $\Lambda^{t-1}\phi$ are zero.
\end{proof}

\begin{lemma}
\label{lemma:valuation}
Let $\frakp$ be as in~\eqref{equation:p}, and let $v$ denote the valuation of the discrete valuation ring~$R_\frakp$. Then, for integers $1\le i_1<i_2<\cdots<i_{t-1}\le n$, the minors of $H$ satisfy
\[
v([1\ \dots\ t-1 \mid i_1\ \dots\ i_{t-1}])\ =\ n+1-i_{t-1}.
\]
Consequently for $\frakp^{\la k\ra}$ as defined earlier, and $k$ with $1\le k\le n-t+2$, one has
\[
\frakp^{\la k\ra}\subseteq\frakp^{(k)}
\quad\text{and}\quad
\frakp^{\la k\ra}R_\frakp=\frakp^{(k)}R_\frakp.
\]
\end{lemma}

\begin{proof}
Set $\pi\colonequals [1\ \dots\ t-1\mid n-t+2\ \dots\ n]$. We will prove inductively that
\begin{equation}
\label{equation:val1}
v([1\ \dots\ t-1 \mid i_1\ \dots\ i_{t-1}])\ =\ (n+1-i_{t-1})v(\pi),
\end{equation}
with the base case for the induction being $i_{t-1}=n$. By Lemma~\ref{lemma:identity}~(1) one has
\begin{equation}
\label{equation:val2}
[1\ \dots\ t-1 \mid i_1\ \dots\ i_{t-1}] \times [2\ \dots\ t \mid n-t+2\ \dots\ n]
\ =\ \pi\times [2\ \dots\ t \mid i_1\ \dots\ i_{t-1}].
\end{equation}
We work in the ring $R_\frakp$, where the minor
\[
[2\ \dots\ t \mid n-t+2\ \dots\ n]
\]
is a unit. If~$i_{t-1}=n$, then $[2\ \dots\ t \mid i_1\ \dots\ i_{t-1}]$ is a unit in $R_\frakp$ as well, and it follows that
\[
v([1\ \dots\ t-1 \mid i_1\ \dots\ i_{t-1}])\ =\ v(\pi),
\]
which proves the base case. For the inductive step, assume that $i_{t-1}<n$ and that~\eqref{equation:val1} holds for larger values of $i_{t-1}$. Since
\[
[2\ \dots\ t \mid i_1\ \dots\ i_{t-1}]\ =\ [1\ \dots\ t-1 \mid i_1+1\ \dots\ i_{t-1}+1],
\]
the inductive hypothesis gives
\[
v([2\ \dots\ t \mid i_1\ \dots\ i_{t-1}])\ =\ (n-i_{t-1})v(\pi).
\]
Combining this with~\eqref{equation:val2}, it follows that
\[
v([1\ \dots\ t-1 \mid i_1\ \dots\ i_{t-1}])\ =\ v(\pi)+(n-i_{t-1})v(\pi)\ =\ (n+1-i_{t-1})v(\pi),
\]
which completes the proof of~\eqref{equation:val1}.

Since the valuation of each minor generating the ideal $\frakp$ is a positive integer multiple of~$v(\pi)$, it follows that $\pi$ generates the maximal ideal of $R_\frakp$, and that $v(\pi)=1$. Lastly, note that the minors that generate the ideal~$\frakp^{\la k\ra}$ are precisely those with valuation at least $k$.
\end{proof}

The following is a slight modification of \cite[Lemma~4]{JWatanabe}, adapted to our notation, and with a shorter proof. 

\begin{lemma}
\label{lemma:watanabe}
Let $R$ be a $t\times n$ Hankel determinantal ring over a field $\FF$. Set
\[
\Delta\colonequals[1\ \dots\ t-1 \mid 1\ \dots\ t-1],
\]
viewed as an element of $R$. Then:
\begin{enumerate}[\ \rm(1)]
\item the ideal $\Delta R$ has radical $\frakp$, for $\frakp$ as in~\eqref{equation:p},
\item $R_\Delta = {\FF[x_1,\ \dots,\ x_{2t-2}}]_\Delta$, and
\item the elements $x_1,\ \dots,\ x_{2t-2}$ of $R$ are algebraically independent over~$\FF$.
\end{enumerate}
\end{lemma}

\begin{proof} (1) In the notation of Lemma~\ref{lemma:identity}, the ideal $\frakp^{\la k\ra}$ is $I_{t-1}(Y(t-1,n-k+1))$, where~$Y$ is the image of the Hankel matrix $H$ in $R$. Since
\[
I_{t-1}(Y(t-1,n-k+1))\ =\ I_{t-1}(Y(t,n-k)),
\]
Lemma~\ref{lemma:identity}~(2) gives
\begin{multline*}
I_{t-1}(Y(t-1,n-k+1))^2\ =\ I_{t-1}(Y(t-1,n-k))\ I_{t-1}(Y(t,n-k+1))\\
\subset \ I_{t-1}(Y(t-1,n-k))
\end{multline*}
i.e.,
\[
(\frakp^{\la k\ra})^2\ \subset\ \frakp^{\la k+1\ra}.
\]
Since $\frakp^{\la 1\ra}=\frakp$ and $\frakp^{\la n-t+2\ra}=\Delta R$, we are done.

(2) For each $a\ge t$, we have $[1\ \dots\ t \mid 1\ \dots\ t-1 \ a]=0$ in $R$, so
\[
x_{t+a-1}\Delta\ \in\ \FF[x_1, \dots, x_{t+a-2}].
\]
Since $\Delta\in\FF[x_1, \dots, x_{t+a-2}]$, it follows that 
\[
{\FF[x_1, \dots, x_{t+a-1}]}_\Delta\ =\ {\FF[x_1, \dots, x_{t+a-2}]}_\Delta.
\]
Iterating the above display, one gets the desired result.

(3) The dimension of $R$ is $2t-2$, hence $\dim \FF[x_1,\ \dots,\ x_{2t-2}]=2t-2$. 
\end{proof} 

\begin{lemma}
\label{lemma:cm}
For each $k$ with $1\le k\le n-t+2$, the ring $R/\frakp^{\la k\ra}$ is Cohen-Macaulay. Hence the ideal $\frakp^{\la k\ra}$ is a maximal Cohen-Macaulay $R$-module; in particular, it is reflexive.
\end{lemma}

\begin{proof}
Since $R/\frakp$ is a polynomial extension of a $(t-1)\times n$ Hankel determinantal ring, its multiplicity is
\[
e(R/\frakp)\ =\ \binom{n}{t-2}.
\]
Fix $k$ with $1\le k\le n-t+2$. Since $\Delta \in\frakp^{\la k\ra}$, it follows from Lemma~\ref{lemma:watanabe} that $\frakp^{\la k\ra}$ has radical~$\frakp$. The associativity formula for multiplicities, \cite[Corollary~4.7.8]{Bruns-Herzog}, then gives the first equality in
\[
e(R/\frakp^{\la k\ra})\ =\ \ell\left(\frac{R_\frakp}{\frakp^{\la k\ra}R_\frakp}\right)e(R/\frakp)\ =\ k\binom{n}{t-2},
\]
while the second equality follows from Lemma~\ref{lemma:valuation}.

Let $A$ be the polynomial ring $\FF[x_1,\dots,x_{n+t-1}]$, and let $P_k$ be the inverse image of $\frakp^{\la k\ra}$ under the canonical surjection $A\to R$. The images of the indeterminates
\[
\bsx\colonequals x_1, \dots, x_{t-2}, x_{n+1}, \dots, x_{n+t-1}
\]
are a homogeneous system of parameters for~$A/P_k = R/\frakp^{\la k\ra}$. Set
\[
J\colonequals P_k+(\bsx)A.
\]
Using, for example, \cite[Corollary~4.7.11]{Bruns-Herzog}, one has
\begin{equation}
\label{equation:inequalities}
\ell(A/J)\ \ge\ e\big(\bsx,R/\frakp^{\la k\ra}\big)\ \ge\ e\big(R/\frakp^{\la k\ra}\big)\ =\ k\binom{n}{t-2}.
\end{equation}
We claim that
\[
\ell(A/J)\ \le\ k\binom{n}{t-2}.
\]
Assuming the claim, all the terms in~\eqref{equation:inequalities} are equal, but then $R/\frakp^{\la k\ra}$ is Cohen-Macaulay using, again, \cite[Corollary~4.7.11]{Bruns-Herzog}.

To prove the claim, consider the degrevlex order on $A$ induced by
\[
x_1>x_2>\cdots>x_{n+t-1}.
\]
Then the initial ideal of $J$ contains the ideal
\[
(\bsx)\ +\ (x_{t-1}, \dots, x_{n-k+1})^{t-1}\ +\ (x_t, \dots, x_n)^t,
\]
so it suffices to verify that the length of
\[
\frac{\FF[x_{t-1},\dots,x_n]}{(x_{t-1}, \dots, x_{n-k+1})^{t-1} + (x_t, \dots, x_n)^t}
\]
is at most
\[
k\binom{n}{t-2}.
\]
This is immediate from the following lemma.
\end{proof}

\begin{lemma}
Let $\FF$ be a field, and consider integers $t\ge 2$ and $1\le r\le s$. Then
\[
\ell\left(\frac{\FF[y_1,\dots,y_s]}{(y_1, \dots, y_r)^{t-1} + (y_2, \dots, y_s)^t}\right)
\ =\ (s-r+1)\binom{s+t-2}{t-2}.
\]
\end{lemma}

\begin{proof}
When $r=1$, the length in question is that of
\[
\frac{\FF[y_1]}{(y_1^{t-1})}\otimes_{\FF}\frac{\FF[y_2,\dots,y_s]}{(y_2, \dots, y_s)^t},
\]
which equals
\[
(t-1)\binom{s-1+t-1}{t-1}\ =\ s\binom{s+t-2}{t-2},
\]
so the asserted formula holds. Assume for the rest that $r\ge 2$.

The case when $t=2$ is readily checked as well; we proceed by induction on $t$ and $s$. Set
\[
V\colonequals\FF[y_1,\dots,y_s] 
\quad\text{and}\quad
I\colonequals (y_1, \dots, y_r)^{t-1} + (y_2, \dots, y_s)^t,
\]
and consider the exact sequence
\[
0\to V/(I:y_2) \to V/I \to V/(I+y_2V) \to 0.
\]
Since $(I:y_2) = (y_1, \dots, y_r)^{t-2} + (y_2, \dots, y_s)^{t-1}$, the inductive hypothesis gives
\begin{align*}
\ell(V/I) &\ =\ \ell(V/(I:y_2)) + \ell(V/(I+y_2V)) \\
&\ =\ (s-r+1)\binom{s+(t-1)-2}{(t-1)-2} + ((s-1)-(r-1)+1)\binom{(s-1)+t-2}{t-2} \\
&\ =\ (s-r+1)\binom{s+t-2}{t-2}.\qedhere
\end{align*}
\end{proof}

\begin{proof}[Proof of Theorem~\ref{theorem:class:group}]
By Lemma~\ref{lemma:watanabe}, the ring $R_\Delta$ is a localization of a polynomial ring, and hence a UFD. Nagata's theorem, e.g.,~\cite[page~315]{Bruns-Herzog}, then implies that $\Cl(R)$ is generated by the height $1$ prime ideals of $R$ that contain $\Delta$, namely by the ideal $\frakp$.

Fix $k$ with $1\le k\le n-t+2$. Then $\frakp^{\la k\ra}$ has radical $\frakp$, and is unmixed by Lemma~\ref{lemma:cm}. Thus, the primary decomposition of $\frakp^{\la k\ra}$ has the form $\frakp^{(i)}$ for some $i$. The integer~$i$ can be computed after localization at $\frakp$, but then Lemma~\ref{lemma:valuation} implies that $\frakp^{\la k\ra}=\frakp^{(k)}$ as claimed. Note that the ideal $\frakp^{\la k\ra}$ is principal precisely when $k=n-t+2$. Lastly, each $\frakp^{\la k\ra}$ is a maximal Cohen-Macaulay module by Lemma~\ref{lemma:cm}.
\end{proof}

\begin{remark}
\label{remark:mcm}
Let $Y$ be a~$t\times n$ matrix of indeterminates over a field $\FF$, and consider the generic determinantal ring
\[
B\colonequals\FF[Y]/I_t(Y).
\]
Set $P$ to be the prime ideal of $B$ generated by the size $t-1$ minors of the first $t-1$ rows of~$Y$, and $Q$ to be the prime generated by the size $t-1$ minors of the first $t-1$ columns. By \cite[Example~9.27(d)]{Bruns-Vetter}, the following are maximal Cohen-Macaulay $B$-modules:
\[
B,\ P,\ Q,\ Q^2,\ \dots,\ Q^{n-t+1},
\]
and, in fact, the only rank one maximal Cohen-Macaulay $B$-modules up to isomorphism. The canonical module of $B$ is isomorphic to $Q^{n-t}$, see \cite[Theorem~8.8]{Bruns-Vetter}.

Since the Hankel determinantal ring $R$ may be obtained as the specialization of $B$ modulo a regular sequence, it follows that the images in $R$ of the modules displayed above are Cohen-Macaulay $R$-modules. Note that~$\frakp=PR$, and set~$\frakq\colonequals QR$, in which case
\[
R,\ \frakp,\ \frakq,\ \frakq^2,\ \dots,\ \frakq^{n-t+1}
\]
are Cohen-Macaulay $R$-modules. Due to the symmetry in a Hankel matrix, one has
\[
\frakq\ =\ \frakp^{\la n-t+1\ra}\ =\ \frakp^{(n-t+1)}.
\]
Fix $i$ with $1\le i\le n-t+1$. Since $\frakq^i$ is a maximal Cohen-Macaulay $R$-module, and hence a divisorial ideal, it follows that
\[
\frakq^i\ =\ (\frakp^{(n-t+1)})^i\ =\ \frakp^{(i(n-t+1))}\ \cong\ \frakp^{(n-t+2-i)}.
\]
In particular, $\frakq^{n-t+1}\cong\frakp$, and the $n-t+3$ rank one maximal Cohen-Macaulay $B$-modules specialize to the $n-t+2$ elements of the divisor class group of $R$.

The canonical module $Q^{n-t}$ of $B$ specializes to the canonical module
\[
\frakq^{n-t}\ \cong\ \frakp^{(2)}
\]
of $R$. Since the $a$-invariant of the ring $R$ is $1-t$, and $\frakp^{(2)}$ is generated in degree $t-1$, it follows that the \emph{graded} canonical module of $R$ is
\[
\omega_R\colonequals\frakp^{(2)}.
\]
Note that the number of generators of $\omega_R$ as an $R$-module is
\[
\binom{n-1}{t-1}.
\]
Since $\omega_R$ is a reflexive $R$-module of rank one, it corresponds to an element $[\omega_R]$ of $\Cl(R)$. The order of this element is
\[
\ord[\omega_R]\ =\ \begin{cases} n-t+2 & \text{ if $n-t+2$ is odd},\\ (n-t+2)/2 & \text{ if $n-t+2$ is even}. \end{cases}
\]
\end{remark}

%%%%%%%%%%%%%%%%%%%%%%%%%%%%%%%%%%%%%%%%%%%%%%%%%%%%%%%%%%%%%%%%%%%%%%%%%%%%%%%%%%
\section{\texorpdfstring{$F$}{F}-purity and the \texorpdfstring{$F$}{F}-pure threshold}
%%%%%%%%%%%%%%%%%%%%%%%%%%%%%%%%%%%%%%%%%%%%%%%%%%%%%%%%%%%%%%%%%%%%%%%%%%%%%%%%%%

Following \cite[page~121]{Hochster-Roberts:purity}, a ring $R$ of positive prime characteristic is \emph{$F$-pure} if the Frobenius endomorphism $F\colon R\to R$ is pure. We prove:

\begin{theorem}
\label{theorem:fpure}
Let $R$ be a Hankel determinantal ring over a field $\FF$. If $\FF$ has positive characteristic, then the ring $R$ is $F$-pure. If $\FF$ has characteristic zero, then $R$ has log canonical singularities.
\end{theorem}

The proof uses the graded version of Fedder's criterion, \cite[Theorem~1.12]{Fedder}, and a result from \cite{Conca}; we record these below:

\begin{theorem}[Fedder's criterion]
\label{theorem:fedder}
Let $A$ be an $\NN$-graded polynomial ring, where $A_0$ is a field of characteristic $p>0$. Let $I$ be a homogeneous ideal of~$A$, and set $R \colonequals A/I$. Let $\frakm$ be the homogeneous maximal ideal of~$A$. Then $R$ is $F$-pure if and only if
\[
(I^{[p]}:_AI)\ \nsubseteq\ \frakm^{[p]}.
\]
\end{theorem}

The following lemma can be seen as a special case of \cite[Theorem~3.12]{Conca} that express the primary decomposition of a product of Hankel determinantal ideals in terms of symbolic powers and the so-called gamma functions. We present here a direct argument that is based only on \cite[Lemma~3.7]{Conca}. 

\begin{lemma}
\label{lemma:symbolic}
Let $A\colonequals\FF[x_1,\dots,x_{s+r-1}]$ be a polynomial ring over a field $\FF$, and let $H$ be the~$r\times s$ Hankel matrix in the indeterminates $x_1,\dots,x_{s+r-1}$. Set $I\colonequals I_t(H)$, where $t$ is an integer with $1\le t\le\min\{r,s\}$. Let $d$ be a positive integer, and let $\delta_1,\dots,\delta_m$ be minors of~$H$ such that $m\le d$ and $\sum_i\deg\delta_i\ge td$. Then
\[
\delta_1\cdots\delta_m\ \in\ I^d.
\]
\end{lemma}

\begin{proof} By adding factors of degree $0$ if needed, we may assume that $m=d$. For $u$ an integer, set~$I_u\colonequals I_u(H)$. If $\deg\delta_i\ge t$ for all $i=1,\dots,d$, then the assertion is obvious. If~$\deg \delta_i< t$ for some $i$, say $u=\deg \delta_1< t$, then, since $\sum_i\deg\delta_i\ge td$, there must be an index~$j$ such that $\deg\delta_j>t$, say $v=\deg\delta_2>t$. By \cite[Lemma 3.7]{Conca} one has
\[
I_uI_v\subseteq I_{u+1}I_{v-1}
\]
since~$u+1<v$. Hence we may replace $\delta_1\delta_2$ in the product with $\delta'_1\delta'_2$, where $\deg\delta'_1=u+1$ and $\deg \delta'_2=v-1$. Repeating the argument as needed, we obtain the desired assertion.
\end{proof} 

\begin{proof}[Proof of Theorem~\ref{theorem:fpure}]
The characteristic zero case follows from the positive characteristic assertion by \cite[Theorem~3.9]{Hara-Watanabe}; in view of this, let $\FF$ be a field of characteristic $p>0$. Set $A\colonequals\FF[x_1,\dots,x_{n+t-1}]$ and $I\colonequals I_t(H)$. By Fedder's criterion, it suffices to verify that
\[
(I^{[p]}:_AI)\ \nsubseteq\ \frakm^{[p]},
\]
where $\frakm$ is the homogeneous maximal ideal of $A$. We construct a polynomial $f$ with
\begin{equation}
\label{equation:symbolic}
f\ \in\ I^{(n-t+1)}
\end{equation}
such that, with respect to the lexicographic order $x_1>x_2> \cdots >x_{n+t-1}$, one has
\[
\inlex(f)\ =\ x_1x_2\cdots x_{n+t-1}.
\]
Since the initial term of $f$ is squarefree, it follows that $f^{p-1}\notin \frakm^{[p]}$. We claim that~\eqref{equation:symbolic} implies $f^{p-1} \in (I^{[p]}:_AI)$, i.e.,
\[
f^{p-1}I\ \subseteq\ I^{[p]}.
\]
By the flatness of the Frobenius endomorphism of $A$, the set of associated primes of $A/I^{[p]}$ equals that of~$A/I$, so it suffices to verify that the containment displayed above holds after localization at the prime ideal $I$. The ideal $I$ has height $n-t+1$, so $(A_I,IA_I)$ is a regular local ring of dimension $n-t+1$, and the pigeonhole principle gives
\[
I^{(n-t+1)(p-1)+1}A_I\ \subseteq\ I^{[p]}A_I.
\]
Using~\eqref{equation:symbolic}, it follows that
\[
f^{p-1}IA_I\ \subseteq\ I^{(n-t+1)(p-1)+1}A_I,
\]
which proves the claim. It remains to construct $f$ with the properties asserted above; the construction depends on whether $n+t-1$ is odd or even:

Suppose $n+t-1$ is odd, set $k\colonequals (n+t)/2$. Then $I$ also equals the ideal generated by the size $t$ minors of the~$k \times k$ Hankel matrix
\[
H'\colonequals\begin{pmatrix}
x_1 & x_2 & x_3& \cdots & x_k\\
x_2 & x_3 & x_4 & \cdots & x_{k+1}\\
\vdots & \vdots & \vdots & \vdots & \vdots\\
x_k & x_{k+1} &\cdots & \cdots & x_{n+t-1}
\end{pmatrix}.
\]
Let $f$ be the product of $\delta_1\colonequals{[1\ \dots\ k \mid 1\ \dots\ k]}_{H'}$ and $\delta_2\colonequals{[1\ \dots\ k-1 \mid 2\ \dots\ k]}_{H'}$. Then
\[
\inlex(f)\ =\ (x_1x_3 \cdots x_{n+t-1})(x_2x_4\cdots x_{n+t-2})\ =\ x_1x_2\cdots x_{n+t-1}
\]
as claimed. Let $\delta_3,\dots,\delta_{n-t+1}$ be size $t-1$ minors of $H'$. Then Lemma~\ref{lemma:symbolic} implies that
\[
\delta_1\cdots\delta_{n-t+1}\ \in\ I^{n-t+1}.
\]
Since $I$ is a prime ideal generated in degree $t$, and each of $\delta_3,\dots,\delta_{n-t+1}$ has degree $t-1$, it follows that $f=\delta_1\delta_2$ belongs to the symbolic power $I^{(n-t+1)}$, as claimed in~\eqref{equation:symbolic}.

When $n+t-1$ is even, set $k\colonequals (n+t-1)/2$, and consider the~$k \times (k+1)$ Hankel matrix
\[
H''\colonequals\begin{pmatrix}
x_1 & x_2 & \cdots & x_k & x_{k+1}\\
x_2 & x_3 & \cdots & x_{k+1} & x_{k+2}\\
\vdots & \vdots & \vdots & \vdots & \vdots\\
x_k & \cdots & \cdots & x_{n+t-2} & x_{n+t-1}
\end{pmatrix}.
\]
Then $I$ equals $I_t(H'')$. Take $f$ to be the product of the minors $\delta_1\colonequals{[1\ \dots\ k \mid 1\ \dots\ k]}_{H''}$ and~$\delta_2\colonequals{[1\ \dots\ k \mid 2\ \dots\ k+1]}_{H''}$, in which case
\[
\inlex(f)\ =\ (x_1x_3 \cdots x_{n+t-2})(x_2x_4\cdots x_{n+t-1})\ =\ x_1x_2\cdots x_{n+t-1}.
\]
Choosing size $t-1$ minors $\delta_3,\dots,\delta_{n-t+1}$ of $H''$, Lemma~\ref{lemma:symbolic} gives
\[
\delta_1\cdots\delta_{n-t+1}\ \in\ I^{n-t+1}
\]
and hence $f\in I^{(n-t+1)}$, as in the previous case.
\end{proof}

The definition of $F$-pure thresholds is due to Takagi and Watanabe \cite{Takagi-Watanabe}, and provides a positive characteristic analogue of the log canonical threshold. We focus here on the $F$-pure threshold of a homogeneous ideal in standard graded $F$-pure ring:

\begin{definition}
Let $A$ be a polynomial ring over an $F$-finite field of characteristic~$p>0$, and let $I$ be a homogeneous ideal such that $R \colonequals A/I$ is $F$-pure. Let $\fraka$ be a homogeneous ideal of~$R$, and let $J$ be its preimage in $A$. Given $e \in \NN$, set
\[
\nu_e(\fraka)\colonequals\max\left\{r \ge 0 \mid (I^{[q]}:_AI)J^r \not\subseteq \frakm_A^{[q]} \right\},
\]
where $q=p^e$. Then the \emph{$F$-pure threshold} of $\fraka\subset R$ is
\[
\fpt(\fraka) \colonequals \lim_{e \to \infty}\nu_e(\fraka)/p^e.
\] 
\end{definition}

Suppose, in addition, that $R$ is normal; let $\omega_R$ be the graded canonical module of $R$. Taking $\fraka$ to be $\frakm_R$ in the above definition, \cite[Theorem~4.1]{STV} implies that $-\nu_e(\frakm_R)$ equals the degree of a minimal generator of $\omega_R^{(1-q)}$. Using this, we obtain:

\begin{theorem}
\label{theorem:fpt1}
Let $R=\FF[x_1,\dots,x_{n+t-1}]/I_t(H)$, where $\FF$ is a field of characteristic $p>0$, and $H$ is a $t\times n$ Hankel matrix. Then the $F$-pure threshold of $\frakm_R\subset R$ is
\[
\fpt(\frakm_R)\ =\ \frac{2(t-1)}{n-t+2}.
\]
\end{theorem}

\begin{proof}
Recall from Remark~\ref{remark:mcm} that $\omega_R=\frakp^{(2)}$. For an integer $q=p^e$, one then has
\[
\omega_R^{(1-q)}\ =\ \frakp^{(2(1-q))}.
\]
Write
\[
2(q-1)=i(n-t+2)+j,
\quad\text{ where }\ 0\le j\le n-t+1.
\]
In view of the graded isomorphism
\[
\frakp^{(n-t+2)}\ \cong\ R(-(t-1)),
\]
one then has
\[
\omega_R^{(1-q)}\ =\ \frakp^{(2(1-q))}\ =\ \frakp^{(-i(n-t+2))}\frakp^{(-j)}\ \cong\ \frakp^{(-j)}(i(t-1)).
\]
Since $0\le j\le n-t+1$, the module $\frakp^{(-j)}$ has minimal generators in degree $0$, which then implies that $\omega_R^{(1-q)}$ has minimal generators in degree $-i(t-1)$, and hence that
\[
\nu_e(\frakm_R)\ =\ i(t-1)\ =\ (t-1) \left\lfloor\frac{2(q-1)}{n-t+2}\right\rfloor.
\]
The calculation of $\fpt(\frakm_R)$ follows immediately from this.
\end{proof}

Using the general theory developed in \cite{Henriques-Varbaro}, one can also compute the $F$-pure threshold of the ideal $I_t(H)$ in the polynomial ring $\FF[x_1,\dots,x_{n+t-1}]$:

\begin{theorem}
\label{theorem:fpt2}
Let $H$ be a $t\times n$ Hankel matrix of indeterminates over a field $\FF$ of positive characteristic. Then the $F$-pure threshold of $I_t(H)\subset \FF[x_1,\dots,x_{n+t-1}]$ is
\[
\fpt(I_t(H))\ =\ \min\left\{\frac{n+t-2i+1}{t-i+1}\mid i=1,\dots,t\right\}.
\]
More precisely, if $\lambda\in \mathbb{R}_{>0}$, the generalized test ideal $\tau(\lambda \bullet I_t(H))$ is
\[
\tau(\lambda \bullet I_t(H))\ =\ \bigcap_{i=1}^tI_i(H)^{(\lfloor\lambda(t-i+1)\rfloor-n-t+2i)}.
\]
\end{theorem}

\begin{proof}
The powers of the ideal $I_t(H)$ are integrally closed by \cite[Theorem~4.5]{Conca}, and using~\cite[Theorem~3.12]{Conca} one has
\[
\bigcup_{s\ge 1}\Ass I_t(H)^s\ \subseteq\ \{I_1(H),\ I_2(H),\ \dots,\ I_t(H)\}.
\]
In the notation of \cite[\S\,3]{Henriques-Varbaro}, by \cite[Theorem~3.12]{Conca} we also infer that $I_t(H)$ satisfies condition ($\diamond$) and that
\[
e_{I_i(H)}(I_t(H))\ =\ t-i+1
\quad\text{ for }\ i=1,\dots,t.
\]

Recall that the polynomial $f$ constructed in the proof of Theorem~\ref{theorem:fpure} has a squarefree initial term. By an argument similar to the one used there for $i=t$, one sees that
\[
f\ \in \ I_i(H)^{(n+t-2i+1)}
\quad\text{ for }\ i=1,\dots,t.
\]
Since $I_i(H)$ is equal to the ideal generated by the size $i$ minors of an $i\times (t+n-i)$ Hankel matrix, it follows that 
$\height I_i(H)=n+t-2i+1$, and that the ideal $I_t(H)$ satisfies the condition ($\diamond+$). The assertion now follows by \cite[Theorem~3.14]{Henriques-Varbaro}.
\end{proof}

\begin{remark}
A similar argument allows one to compute the $F$-pure threshold and the generalized test ideals (in positive characteristic), as well as the log canonical threshold and the multiplier ideals (in characteristic zero), of any product of ideals of minors of a Hankel matrix in a polynomial ring.
\end{remark}

We conclude with the following question; we prove in \cite{CMSV} that the answer is affirmative in a number of cases.

\begin{question}
Is every Hankel determinantal ring over a field of positive characteristic an~$F$-regular ring?
\end{question}

%%%%%%%%%%%%%%%%%%%%%%%%%%%%%%%%%%%%%%%%%%%%%%%%%%%%%%%%%%%%%%%%%%%%%%%%

\end{document}